\documentclass[mathpazo]{csam}

\newcommand{\ds}{\displaystyle}
\usepackage{color}
\begin{document}
\title{Mathematical and numerical analysis to shrinking-dimer saddle dynamics with local Lipschitz conditions}


 \author[L. Zhang et~al.]{Lei Zhang\affil{1},
       Pingwen Zhang\affil{2}~and Xiangcheng Zheng\affil{3}\comma\corrauth}
 \address{\affilnum{1}\ Beijing International Center for Mathematical Research, Center for Quantitative Biology, Peking
University, Beijing 100871, China. \\
           \affilnum{2}\ School of Mathematical Sciences, Laboratory of Mathematics and Applied Mathematics, Peking
University, Beijing 100871, China.
     \\
           \affilnum{3}\ School of Mathematical Sciences, Peking University, Beijing 100871, China. }
 \emails{{\tt zhangl@math.pku.edu.cn} (L.~Zhang), {\tt pzhang@pku.edu.cn} (P.~Zhang),
          {\tt zhengxch@math.pku.edu.cn} (X.~Zheng)}

\begin{abstract}
We present a mathematical and numerical investigation to the shrinking-dimer saddle dynamics for finding any-index saddle points in the solution landscape. Due to the dimer approximation of Hessian in saddle dynamics, the local Lipschitz assumptions and the strong nonlinearity for the saddle dynamics, it remains challenges for delicate analysis, such as the the boundedness of the solutions and the dimer error.
We address these issues to bound the solutions under proper relaxation parameters, based on which we prove the error estimates for numerical discretization to the shrinking-dimer saddle dynamics by matching the dimer length and the time step size. Furthermore, the Richardson extrapolation is employed to obtain a high-order approximation.

 The inherent reason of requiring the matching of the dimer length and the time step size lies in that the former serves a different mesh size from the later, and thus the proposed numerical method is close to a fully-discrete numerical scheme of some space-time PDE model with the Hessian in the saddle dynamics and its dimer approximation serving as a ``spatial operator'' and its discretization, respectively, which in turn indicates the PDE nature of the saddle dynamics. 
\end{abstract}

\ams{37M05, 37N30, 65L20
}
\keywords{Saddle dynamics, solution landscape, saddle points, local Lipschitz condition, error estimate, Richardson extrapolation.}
\maketitle

\section{Introduction}
One of the major challenges in computational physical and chemistry is how to efficiently calculate saddle points on a complicated energy landscape. In comparison with finding local minima, the computation of saddle points is generally more difficult due to their unstable nature. Nevertheless, saddle points provide important information about the physical and chemical properties. For instance, the index-1 saddle point represents the transition states connecting two local minima according to the transition state theory \cite{Mehta,npj2016}, and the index-2 saddle points are particularly interesting in chemical systems for providing valuable information on the trajectories of chemical reactions \cite{Heidrich1986}. The applications of saddle points include nucleation in phase transformations \cite{EV2010,ZhangChe,ZhangCheDu}, transition rates in chemical reactions and computational biology \cite{Han2019transition,HanXu,Nie2020,wang2021modeling,Yin2020nucleation}, etc.

The saddle points can be classified by the (Morse) index, which is characterized by the maximal dimension of a subspace on which the Hessian $H(x)$ is negative definite, according to the Morse theory \cite{Milnor}. 
Most existing searching algorithms focus on finding the index-1 saddle points, e.g. \cite{All,Doye,EZho,Farr,Gou,Henke,Lev,
ZhaDu}. However, the computation of high-index (index$>1$) saddle points receive less attention despite of the fact that the number of high-index saddles are much larger than the number of local minima and index-1 saddles on the complicated energy landscapes \cite{Chen2004,Li2001}.

The original saddle dynamics (SD) aims to find an index-k ($1\leq k\in\mathbb N$) saddle point of an energy function $E(x)$ \cite{YinSISC}
\begin{equation}\label{Sadk}
\left\{
\begin{array}{l}
\ds \frac{dx}{dt} =\beta\bigg(I -2\sum_{j=1}^k v_jv_j^\top \bigg)F(x),\\[0.075in]
\ds \frac{dv_i}{dt}=\gamma \bigg( I-v_iv_i^\top-2\sum_{j=1}^{i-1}v_jv_j^\top\bigg)H(x)v_i,~~1\leq i\leq k.
\end{array}
\right.
\end{equation}
Here the natural force $F:\mathbb R^N\rightarrow \mathbb R^N$ is generated from an energy $E(x)$ by $F(x)=-\nabla E(x)$, $H(x):=-\nabla^2 E(x)$ corresponds to the Hessian of $E(x)$, $\beta$, $\gamma>0$ are relaxation parameters, $x$ represents the position variable and direction variables $\{v_i\}_{i=1}^k$ form a basis for the unstable subspace of the Hessian at $x$. 

Because the Hessians are often expensive to calculate and store, one can apply first derivatives to approximate the Hessians in Eq. \eqref{Sadk} by using $k$ dimers centered at $x$. To be specific, $H(x) v_i$ is approximated by
\begin{equation}\label{Hzz}
\hat H(x,v_i,l):=\frac{F(x+lv_i)-F(x-lv_i)}{2l}
\end{equation}
with the direction $v_i$ and the dimer length $2l$ for some $l>0$. 

Following the idea of the shrinking dimer dynamics \cite{ZhaDu, ZhaSISC}, we obtain the shrinking-dimer saddle dynamics (SSD)\cite{YinSISC} as follows
\begin{equation}\label{sadk}
\left\{
\begin{array}{l}
\ds \frac{dx}{dt} =\beta\bigg(I -2\sum_{j=1}^k v_jv_j^\top \bigg)F(x),\\[0.075in]
\ds \frac{dv_i}{dt}=\gamma \bigg( I-v_iv_i^\top-2\sum_{j=1}^{i-1}v_jv_j^\top\bigg)\hat H(x,v_i,l),~~1\leq i\leq k,\\
\ds \frac{dl}{dt}=-l.
\end{array}
\right.
\end{equation}

By using the SSD method as a key ingredient, the solution landscape can be constructed by connecting the high-index saddle points to low-index saddle points and local minima \cite{YinPRL,YinSCM}. The solution landscape serves as an efficient approach to provide a global structures of all stationary points of the model systems and has been widely applied in various fields \cite{HanYin, Han2021,Xu_PRE,YinPRL,Yin2020nucleation, Yin2022, YuZhaZha}. 

Despite the growing applications of the SSD, the corresponding mathematical and numerical analysis are still far from well-developed. Most existing works only focus on the numerical analysis of the index-1 SD in recent years  \cite{Gao,Gou,Lev,ZhaDu}, which corresponds to (\ref{Sadk}) with $k=1$, and the corresponding results for (shrinking-dimer) high-index SD are meager. 
In a very recent work \cite{Z3}, numerical discretization to SD (\ref{Sadk}) was analyzed, the proof of which depends heavily on the global Lipschitz assumptions of both $F(x)$ and $H(x)$. However, $F(x)$ and $H(x)$ generally have complex nonlinear forms that only admit local Lipschitz conditions. Furthermore, to avoid the direct calculation of Hessians, the SSD is usually used instead of the SD in practice. But, the dimer approximation of $H(x)$ in Eq. (\ref{sadk}) introduces additional errors, which generate significant differences from the numerical analysis of the SD \cite{Z3} and lead to the failure of the error estimates therein. Moreover, as the dimer length serves like a ``step size'' in the dimer approximation, it needs to be carefully chosen in order to match the time step size. 

Motivated by these discussions, in this work we aim to prove the boundedness of the exact solutions and optimal-order error estimates of the numerical discretization to the SSD (\ref{sadk}) with respect to the time step size. Due to the strong nonlinearity of the system and the local Lipschitz conditions, the boundedness of solutions is proved under some restrictions of the relaxation parameters (cf. (\ref{beta})). Based on the proposed first-order scheme, the Richardson extrapolation is further developed to obtain a high-order approximation. As the dimer length serves as a different mesh size from the time step size, the proposed numerical method is close to a fully-discrete numerical scheme of some space-time PDE model with the Hessian and its dimer approximation serving as a `` spatial operator '' and its discretization, respectively, which in turn indicates the PDE nature of the saddle dynamics. 

The rest of the paper is organized as follows: In Section 2 we estimate $x(t)$ and $\{v_i(t)\}_{i=1}^k$ in (\ref{sadk}) under local Litschitz conditions, which supports the subsequent numerical analysis. In Section 3 we present the numerical scheme of the SSD (\ref{sadk}) and prove auxiliary estimates for the sake of the error estimates for the numerical discretization in Section 4. We also propose the Richardson extrapolation in this section to obtain a high-order approximation. In Section 5 we extend the developed techniques to numerically analyze the generalized SSD for non-gradient systems. Numerical experiments are performed in Section 6 and we finally address a conclusion in the last section.

\section{Estimate of solutions under local Lipschitz conditions}\label{secxbnd}
In this section we consider the SSD (\ref{sadk}) on $[0,T]$, closed by the following initial conditions \begin{equation}\label{inicondk}
\ds x(0)=x_0,~~v_i(0)=v_{i,0}\text{ for }1\leq i\leq k,~~v_{i,0}^\top v_{j,0}=\delta_{i,j}, ~~l(0)=l_0.
\end{equation}
We will show that, under the local Lipschitz conditions of $F(x)$ and $H(x)$, $x(t)$ and $\{v_i(t)\}_{i=1}^k$ with $0\leq t\leq T$ are bounded (under suitable relaxation parameters) such that in subsequent proofs, we could use the Lipschitz continuity of $F(x)$ and $H(x)$ with a fixed Lipschitz constant, just like imposing the global Lipschitz conditions as in \cite{Z3}. Furthermore, the boundedness of $\{v_i(t)\}_{i=1}^k$ will be used in numerical analysis. 

Let $\|\cdot\|$ be the standard $l^2$ norm of the matrix or the vector. For the sake of the analysis, we make the following assumption throughout the paper:\\
\textbf{Assumption $\mathcal A$}: $F(x)$ and $H(x)$ satisfy local Lipschitz conditions, that is, for any $r>0$ there exists a constant $L_r>0$ such that for $x_1,x_2\in B_r:=\{x\in \mathbb R^N:\|x\|\leq r\}$
\begin{equation}\label{Lip}
 \|F(x_2)-F(x_1)\|\leq L_r\|x_2-x_1\|,~~\|H(x_2)-H(x_1)\|\leq L_r\|x_2-x_1\|. 
 \end{equation}
\subsection{Properties of auxillary functions}
Based on the Assumption $\mathcal A$, we derive some important properties for the following nonlinear functions
 $$\begin{array}{l}
 \ds  X(x,v_1,\cdots,v_k):=\bigg(I -2\sum_{j=1}^k v_jv_j^\top \bigg)F(x),\\[0.15in]
 \ds V_i(x,v_1,\cdots,v_k,l):=\bigg( I-v_iv_i^\top-2\sum_{j=1}^{i-1}v_jv_j^\top\bigg)\hat H(x,v_i,l),~~1\leq i\leq k,
\end{array}  $$
which are indeed right-hand side terms of (\ref{sadk}) without relaxation parameters, in the following theorem.
\begin{theorem}
Under the Assumption $\mathcal A$, for any fixed $r>0$ there exist positive constants $Q_0=Q_0(r)$ and $Q_2=Q_2(r)$ depending on $r$, $L_r$, $k$, $l_0$, $F$ and $H$ such that for $(x,v_1,\cdots,v_k)$, $(\bar x,\bar v_1,\cdots,\bar v_k)\in\mathcal B_r$
\begin{equation}\label{thme1}
\begin{array}{l}
\ds \|X(x,v_1,\cdots,v_k)-X(\bar x,\bar v_1,\cdots,\bar v_k)\|\\[0.05in]
\ds\quad\leq Q_0(r)\big\|x-\bar x,v_1-\bar v_1,\cdots,v_k-\bar v_k\big\|,\\[0.05in]
 \ds\|V_i(x,v_1,\cdots,v_k,l)-V_i(\bar x,\bar v_1,\cdots,\bar v_k,l)\|\\[0.05in]
 \ds\quad\leq Q_2(r)\big(\big\|x-\bar x,v_1-\bar v_1,\cdots,v_k-\bar v_k\big\|+l_0^2\big),~~1\leq i\leq k.
\end{array}
\end{equation}
Here the convex set $\mathcal B_r$ and the norm $\|x,v_1,\cdots,v_k\|$ are defined by 
 $$\mathcal B_r:=\bigg\{(x,v_1,\cdots,v_k):\|x,v_1,\cdots,v_k\|:=\bigg(\|x\|^2+\sum_{i=1}^k\|v_i\|^2\bigg)^{1/2}\leq r\bigg\}.$$
\end{theorem}
\begin{remark}
We write $Q_0$ and $Q_2$ as $Q_0(r)$ and $Q_2(r)$ in order to highlight their dependence on $r$. We neglect their dependence on $k$, $l_0$, $F$ and $H$ in the notations as these are fixed data throughout the paper. 
\end{remark}
\begin{proof}
 Direct calculations show that for $(x,v_1,\cdots,v_k),(\bar x,\bar v_1,\cdots,\bar v_k)\in\mathcal B_r$
\begin{equation}\label{XX}
\begin{array}{l}
\ds \|X(x,v_1,\cdots,v_k)-X(\bar x,\bar v_1,\cdots,\bar v_k)\|\\[0.05in]
\ds\quad\leq \bigg\|\bigg(I -2\sum_{j=1}^k v_jv_j^\top \bigg)(F(x)-F(\bar x))\bigg\|\\[0.15in]
\ds
\qquad+2\bigg\|\bigg(\sum_{j=1}^k\bar v_j\bar v_j^\top-\sum_{j=1}^k v_jv_j^\top \bigg)F(\bar x)\bigg\|\\[0.2in]
\ds\quad\leq (1+2kr^2)L_r\|x-\bar x\|\\[0.05in]
\ds\qquad+2\bigg\|\sum_{j=1}^k \bar v_j(\bar v_j^\top-v_j^\top)+(\bar v_j-v_j)v_j^\top\bigg\|(\|F(0)\|+L_r\|\bar x\|)\\[0.2in]
\ds\quad\leq Q_0(r)\big\|x-\bar x,v_1-\bar v_1,\cdots,v_k-\bar v_k\big\|
\end{array}
\end{equation}
 where 
 $$Q_0(r):=\sqrt{k+1}\max\big\{(1+2kr^2)L_r,4r(\|F(0)\|+L_rr) \big\}. $$
 
 To estimate $V_i(x,v_1,\cdots,v_k,l)-V_i(\bar x,\bar v_1,\cdots,\bar v_k,l)$, we follow \cite[Equation 4]{Gou} to obtain
\begin{equation}\label{H0}
\begin{array}{l}
\ds \hat H(x,v_i,l)=\frac{F(x+lv_i)-F(x-lv_i)}{2l}= H(x)v_{i}+O(l^2),~~O(l^2)\leq Q_1(r)l^2,
\end{array}
\end{equation}
 and consequently,
 $$\big\|\hat H(x,v_i,l)-\hat H(\bar x,\bar v_i,l)\big\|\leq \big(\|H(0)\|+rL_r\big)\big(\|v_i-\bar v_i\|+\|x-\bar x\|\big)+2Q_1(r)l_0^2. $$
We apply this relation and a similar derivation as (\ref{XX}) to obtain the second equation of (\ref{thme1}). Thus we complete the proof.
\end{proof}

 \subsection{Estimate of $(x,v_1,\cdots,v_k)$}
Let $\varepsilon_0>0$ be a fixed constant. By Assumption $\mathcal A$, for a fixed $r_0$ satisfying
\begin{equation}\label{r}
 r_0\geq \|x_0,v_{1,0},\cdots,v_{k,0}\|+\varepsilon_0
 \end{equation}
 there exists a constant $L_{r_0}>0$ such that (\ref{Lip}) is satisfied. We then define $\tilde X(x,v_1,\cdots,v_k)$ and $\tilde V_i(x,v_1,\cdots,v_k,l)$ for $1\leq i\leq k$ such that 
 \begin{itemize}
 \item[(i)] $\tilde X=X,~~\tilde V_i=V_i,~~1\leq i\leq k,~~(x,v_1,\cdots,v_k)\in\mathcal B_{r_0}$;
 
 \item[(ii)] $\tilde X$ and $\tilde V_i$ for $1\leq i\leq k $ satisfy the conditions (\ref{thme1}) globally (i.e., for any choice of $(x,v_1,\cdots,v_k)$ and $(\bar x,\bar v_1,\cdots,\bar v_k)$) with respect to the fixed constants $Q_0(r_0)$ and $Q_2(r_0)$.
 \end{itemize}
 \begin{remark}
 A possible choice of $\tilde X$ is
 $$\tilde X(x,v_1,\cdots,v_k)=\left\{\begin{array}{l}
 X(x,v_1,\cdots,v_k),\qquad\qquad\qquad\quad(x,v_1,\cdots,v_k)\in \mathcal B_{r_0};\\[0.1in]
 \ds X(\lambda x,\lambda v_1,\cdots,\lambda v_k),~~\lambda:=\frac{r}{\|x,v_1,\cdots,v_k\|},\text{  otherwise.}
\end{array} \right. $$
$\tilde V_1,\cdots,\tilde V_k$ could be similarly defined.
 \end{remark}
  Consider the following modified SSD on $[0,T]$ with $X$ and $\{V_i\}$ in (\ref{sadk}) replaced by $\tilde X$ and $\{\tilde V_i\}$, respectively
\begin{equation}\label{sadkaux}
\left\{
\begin{array}{l}
\ds \frac{d\tilde x}{dt} =\beta\tilde X,\\[0.1in]
\ds \frac{d\tilde v_i}{dt}=\gamma \tilde V_i,~~1\leq i\leq k,\\[0.1in]
\ds \frac{dl}{dt}=-l,
\end{array}
\right.
\end{equation}
equipped with the initial conditions (\ref{inicondk}). 
 We multiply $\tilde x$ on both sides of the first equation of (\ref{sadkaux}) and integrate the resulting equation from $0$ to $t$ to get
$$\begin{array}{rl}
\ds \|\tilde x(t)\|^2&\hspace{-0.1in}\ds\leq \|x_0\|^2+2\beta\int_0^t\tilde x^\top \tilde X ds\\
&\ds\hspace{-0.1in}\leq\|x_0\|^2+2\beta\int_0^t\|\tilde x(s)\|\big[\|\tilde X( x_0,v_{1,0},\cdots,v_{k,0})\|\\[0.15in]
&\ds\hspace{-0.1in}
\quad+Q_0(r_0)\big(\|\tilde x(s),\tilde v_1(s),\cdots,\tilde v_k(s)\|+\|x_0,v_{1,0},\cdots,v_{k,0}\|\big)\big]ds\\[0.1in]
&\hspace{-0.1in}\ds\leq \|x_0\|^2+(\beta+Q_0(r_0)\beta)\int_0^t\|\tilde x(s)\|^2ds+\beta T\big(\|X(x_0,v_{1,0},\cdots,v_{k,0})\|+Q_0(r_0)r_0\big)^2\\[0.05in]
&\hspace{-0.1in}\ds\quad+\beta Q_0(r_0)\int_0^t \|\tilde x(s),\tilde v_1(s),\cdots,\tilde v_k(s)\|^2 ds
\end{array} $$
where we used $(x_0,v_{1,0},\cdots,v_{k,0})\in\mathcal B_{r_0}$ in this derivation. 
 Similarly we bound $\tilde v_i$ in (\ref{sadkaux}) by
$$\begin{array}{rl}
\ds \|\tilde v_i(t)\|^2&\hspace{-0.1in}\ds\leq \|v_{i,0}\|^2+(\gamma+Q_2(r_0)\gamma)\int_0^t\|\tilde v_i(s)\|^2ds\\
&\ds\quad+\gamma T\big(\| V_i(x_0,v_{1,0},\cdots,v_{k,0})\|+Q_2(r_0)r_0+Q_2(r_0)l_0^2\big)^2\\[0.05in]
&\ds\quad+\gamma Q_2(r_0)\int_0^t \|\tilde x(s),\tilde v_1(s),\cdots,\tilde v_k(s)\|^2 ds.
\end{array} $$
Furthermore, we apply (\ref{H0}) to obtain
$$\begin{array}{l}
\ds \|X(x_0,v_{1,0},\cdots,v_{k,0})\|\leq \bigg(1+2k\sum_{j=1}^k\|v_{j,0}\|^2\bigg)\|F(x_0)\|=:Q_3,\\
\ds \|V_i(x_0,v_{1,0},\cdots,v_{k,0})\|\leq \bigg(1+\|v_{i,0}\|^2+ 2k\sum_{j=1}^k\|v_{j,0}\|^2\bigg)\\[0.15in]
\ds\hspace{1.7in}\times\big(\|H(x_0)\|\|v_{i,0}\|+Q_1(r_0)l_0^2\big)=:Q_4(r_0).
\end{array} $$
We incorporate the above equations to obtain
$$\begin{array}{l}
\ds \|\tilde x(t),\tilde v_1(t),\cdots,\tilde v_k(t)\|^2\leq \|x_0,v_{1,0},\cdots,v_{k,0}\|^2\\[0.05in]
\ds\qquad\qquad+Q_5(r_0)+Q_6(r_0)\int_0^t \|\tilde x(s),\tilde v_1(s),\cdots,\tilde v_k(s)\|^2 ds
\end{array}   $$
where 
$$\begin{array}{l}
\ds Q_5(r_0):=\beta T\big(Q_3+Q_0(r_0)r_0\big)^2+k\gamma T\big(Q_4(r_0)+Q_2(r_0)r_0+Q_2(r_0)l_0^2\big)^2,\\[0.05in]
\ds Q_6(r_0):=\max\{\beta(1+Q_0(r_0)),\gamma(1+Q_2(r_0))\}+\beta Q_0(r_0)+k\gamma Q_2(r_0).
\end{array}  $$
 Then an application of the Gronwall's inequality yields
$$\begin{array}{l}
\ds \|\tilde x(t),\tilde v_1(t),\cdots,\tilde v_k(t)\|\leq \big(\|x_0,v_{1,0},\cdots,v_{k,0}\|^2+Q_5(r_0)\big)^{1/2}e^{Q_6(r_0)T/2}=:S(\beta,\gamma)
\end{array}   $$
for $0\leq t\leq T$. As $S(\beta,\gamma)$ is an increasing function with respect to both $\beta$ and $\gamma$ ($\beta$, $\gamma\geq 0$) and attains its minimum $\|x_0,v_{1,0},\cdots,v_{k,0}\|$ at $\beta=\gamma=0$, we could select $\beta$ and $\gamma$ such that
 \begin{equation}\label{beta}
S(\beta,\gamma)\leq \|x_0,v_{1,0},\cdots,v_{k,0}\|+\varepsilon_0,
 \end{equation}
 which implies
 \begin{equation}\label{xbnd}
 \|\tilde x(t),\tilde v_1(t),\cdots,\tilde v_k(t)\|\leq \|x_0,v_{1,0},\cdots,v_{k,0}\|+\varepsilon_0,~~0\leq t\leq T.
 \end{equation}
 Recall that $r_0\geq \|x_0,v_{1,0},\cdots,v_{k,0}\|+\varepsilon_0$, we base on (i) to conclude that the modified SSD (\ref{sadkaux}) is indeed equivalent to the original SSD (\ref{sadk}) for $(\beta,\gamma)$ satisfying (\ref{beta}). We summarize the findings in the following theorem.
 \begin{theorem}\label{thmxbnd}
Suppose the Assumption $\mathcal A$ holds and $(\beta,\gamma)$ satisfy (\ref{beta}), then $(x,v_1,\cdots,v_k)$ in the SSD (\ref{sadk}) are bounded as (\ref{xbnd}) and thus we could always apply the Lipschitz conditions of $F(x)$ and $H(x)$ in subsequent proofs with a fixed Lipschitz constant $L_{r_0}$ corresponding to the $r_0$ given by (\ref{r}).
 \end{theorem}


\section{Discrete SSD and auxiliary results}\label{secnumer}
In this section we present the numerical scheme to (\ref{sadk}) and prove auxiliary lemmas to be used in the error estimates.
\subsection{Numerical scheme}
 Let $0=t_0<t_1<\cdots<t_K=T$ be a uniform temporal partition of $[0,T]$ with the time step size $\tau:=T/K$ for some $0<K\in\mathbb N$. We approximate the first-order derivative by the Euler scheme at $t_n$ as follows
$$\frac{dg(t_n)}{dt}=\frac{g(t_n)-g(t_{n-1})}{\tau}+R_n^g $$
where $g$ refers to $x$ or $v_i$, and we suppose the truncation error satisfies $\|R_n^g\|=O(\tau)$. Invoking this discretization in (\ref{sadk}) yields the following reference equations for the dynamics (\ref{sadk})
\begin{equation}\label{Refk}
\hspace{-0.05in}\left\{
\begin{array}{l}
\ds\!\!\! x(t_{n}) =x(t_{n-1})+\tau\beta\bigg(I -2\sum_{j=1}^k v_j(t_{n-1})v_j^\top(t_{n-1}) \bigg)F(x(t_{n-1}))+\tau R_n^x,\\[0.075in]
\ds\!\!\! v_i(t_{n})=v_i(t_{n-1})+\tau\gamma\bigg( I-v_i(t_{n-1})v_i^\top(t_{n-1})\\
\ds\hspace{1.5in}-2\sum_{j=1}^{i-1}v_j(t_{n-1})v_j^\top(t_{n-1})\bigg)\\[0.175in]
\ds\hspace{1.1in}\times \hat H(x(t_{n-1}),v_i(t_{n-1}),l(t_{n-1}))+\tau R_{n}^{v_i},~~1\leq i\leq k,\\[0.075in]
\ds\!\!\! l(t_n)=e^{-t_n}l_0
\end{array}
\right.
\end{equation}
where we analytically solved the equation of $l$ without approximation. In the rest of the paper we denote 
$$l_n=l(t_n)=e^{-t_n}l_0$$
 for simplicity and we then drop the truncation errors in the reference equations to obtain the explicit scheme of (\ref{sadk}) 
\begin{equation}\label{FDsadk}
\left\{
\begin{array}{l}
\ds x_{n} =x_{n-1}+\tau\beta\bigg(I -2\sum_{j=1}^k v_{j,n-1}v_{j,n-1}^\top \bigg)F(x_{n-1}),\\[0.1in]
\ds \tilde v_{i,n}=v_{i,n-1}+\tau\gamma\bigg( I-v_{i,n-1}v_{i,n-1}^\top\\[0.1in]
\ds\hspace{0.7in}-2\sum_{j=1}^{i-1}v_{j,n-1}v_{j,n-1}^\top\bigg)\hat H(x_{n-1},v_{i,n-1},l_{n-1}),~~1\leq i\leq k,\\[0.2in]
\ds \{v_{i,n}\}_{i=1}^k=\text{GS}(\{\tilde v_{i,n}\}_{i=1}^k)
\end{array}
\right.
\end{equation}
for $1\leq n\leq K$, equipped with the initial conditions (\ref{inicondk}). Here the notation GS$(\cdot)$ refers to the Gram-Schmidt orthonormalization procedure, the purpose of which is to preserve the orthonormal property of the vectors \cite{YinSISC,YinSCM}. Due to the orthonormalization, $\|v_{i,n}\|=1$ for all possible $i$ and $n$, and by a discrete analogue of the derivations in Section \ref{secxbnd}, we could obtain the estimate of $x_n$. To be specific, let $r_0\geq \|x_0\|+\varepsilon_0$ be a fixed constant for some $\varepsilon_0>0$. Then we consider the following auxiliary problem 
\begin{equation}\label{xtilde}
 \tilde x_{n} =\tilde x_{n-1}+\tau\beta\bigg(I -2\sum_{j=1}^k v_{j,n-1}v_{j,n-1}^\top \bigg)\tilde F(\tilde x_{n-1}) \end{equation}
for $1\leq n\leq K$ with $\tilde x_0=x_0$. Here $\tilde F(\cdot)$ is defined as
$$\tilde F(x)=\left\{\begin{array}{ll}
 F(x),&\ds \|x\|\leq r_0;\\[0.1in]
 \ds F\bigg(\frac{r_0}{\|x\|} x\bigg),&\ds\text{otherwise}
\end{array} \right. $$
such that $\tilde F$ satisfies the global Lipschitz condition with the Lipschitz constant $L_{r_0}$. Then we apply the norm-preserving property of the Householder matrix
$$I -2\sum_{j=1}^k v_{j,n-1}v_{j,n-1}^\top  $$
on (\ref{xtilde}) to obtain
\begin{equation*}
\|\tilde x_n\|\leq \|\tilde x_{n-1}\|+\tau\beta\|\tilde F(\tilde x_{n-1})\|.
\end{equation*}
We then apply the global Lipschitz condition of $\tilde F$ to get
\begin{equation*}
\|\tilde x_n\|\leq \|\tilde x_{n-1}\|+\tau\beta\big(\|\tilde F(0)\|+L_{r_0}\|\tilde x_{n-1}\|\big)= \|\tilde x_{n-1}\|+\tau\beta\big(\|F(0)\|+L_{r_0}\|\tilde x_{n-1}\|\big).
\end{equation*}
Adding this equation for $1\leq n\leq m$ for some $m\leq K$ yields
\begin{equation*}
\|\tilde x_m\|\leq \|x_{0}\|+\beta T\|F(0)\|+\tau\beta L_{r_0}\sum_{n=1}^m\|\tilde x_{n-1}\|.
\end{equation*}
Then an application of the discrete Gronwall inequality leads to
$$\|\tilde x_m\|\leq \big(\|x_{0}\|+\beta T\|F(0)\|\big)e^{\beta L_{r_0}T}. $$
Consequently, if $\beta$ satisfies 
\begin{equation}\label{beta2}
\big(\|x_0\|+\beta T\|F(0)\|\big)e^{\beta L_{r_0} T}\leq \|x_0\|+\varepsilon_0,
\end{equation}
$\tilde x_n$ is bounded as
\begin{equation}\label{xnnorm}
 \|\tilde x_n\|\leq \|x_0\|+\varepsilon_0,~~0\leq n\leq K, 
 \end{equation}
and thus the equation (\ref{xtilde}) is equivalent to the first equation of (\ref{FDsadk}). This implies that $x_n$ is bounded as (\ref{xnnorm}) and we could always apply the Lipschitz conditions of $F(x)$ and $H(x)$ in subsequent proofs with a fixed Lipschitz constant $L_{r_0}$. 

In the rest of the paper we use $Q$ to denote a generic positive constant that may assume different values at different occurrences.
\subsection{Auxiliary estimates}
We prove several auxiliary estimates to support the error estimates. 
 \begin{lemma}\label{lem1k}
 Suppose  (\ref{beta}), (\ref{beta2}) and the Assumption $\mathcal A$ hold and $l_0^2=O(\tau)$, then the following estimates hold for $1\leq n\leq K$ 
 $$\begin{array}{c}
 \ds \big|\tilde v_{m,n}^\top \tilde v_{i,n}\big|\leq Q\tau^2,~~1\leq m<i\leq k.
\end{array}  $$
Here the positive constant $Q$ is independent from $n$, $K$ and $\tau$. 
 \end{lemma}
 \begin{remark}
 Note that the initial value $l_0$ of the dimer length $l$ is chosen in the magnitude of $\sqrt{\tau}$, which is key in preserving the first-order accuracy of the scheme (\ref{FDsadk}) as we will see later. The inherent reason is that the numerical method (\ref{FDsadk}) is close to a fully-discrete numerical scheme of some space-time PDE model with the Hessian and its dimer approximation serving as a ``spatial operator'' and its discretization, respectively, and the  dimer length serves like a ``spatial mesh size'', which should match the time-stepping size $\tau$ to keep the $O(\tau)$ accuracy of the numerical method.   
 \end{remark}
\begin{proof} 
For $1\leq m< i\leq k$ we apply (\ref{H0}) and the symmetry of $H(x)$ to obtain
\begin{equation}\label{Hest}
\ds\big| v^\top_{i,n-1}\hat H(x_{n-1},v_{m,n-1},l_{n-1})-v_{m,n-1}^\top \hat H(x_{n-1},v_{i,n-1},l_{n-1})\big|=\big|O(l^2_{n-1})\big|\leq Q\tau.
 \end{equation}
By the Assumption $\mathcal A$ we conclude that $H(x_n)$ ($0\leq n\leq K$) is bounded, which, together with (\ref{H0}), implies that 
\begin{equation}\label{bndh}
\hat H(x_n,v_{j,n},l_n) \text{ for }0\leq n\leq K\text{ and } 1\leq j\leq k\text{ are  bounded}.
\end{equation}
 We invoke these boundedness and (\ref{Hest}) to the right-hand side of the following equation
\begin{equation}\label{mhback}
\ds \tilde v_{m,n}^\top \tilde v_{i,n}=\tau\gamma\big(v^\top_{i,n-1}\hat H(x_{n-1},v_{m,n-1},l_{n-1})-v_{m,n-1}^\top \hat H(x_{n-1},v_{i,n-1},l_{n-1})\big)+O(\tau^2)
\end{equation}
to complete the proof.
\end{proof}
\begin{lemma}\label{lem1kk}
Under (\ref{beta}), (\ref{beta2}) and the Assumption $\mathcal A$, the following estimates hold for $1\leq n\leq K$ 
 $$ \big|\|\tilde v_{i,n}\|^2-1\big|\leq Q\tau^2,~~1\leq i\leq k.
  $$
Here the positive constant $Q$ is independent from $n$, $K$ and $\tau$. 
 \end{lemma}
\begin{proof}
By the boundedness of $\hat H$ in (\ref{bndh}), we obtain from the second equation of (\ref{FDsadk}) that
\begin{equation}\label{aux1}
\begin{array}{l}
\ds \|\tilde v_{i,n}-v_{i,n-1}\|= \tau\gamma\bigg\|\bigg( I-v_{i,n-1}v_{i,n-1}^\top\\[0.15in]
\ds\hspace{1in}-2\sum_{j=1}^{i-1}v_{j,n-1}v_{j,n-1}^\top\bigg)\hat H(x_{n-1},v_{i,n-1},l_{n-1})\bigg\|\leq Q\tau.
\end{array}
\end{equation}
We multiply $v^\top_{i,n-1}$ on both sides of the second equation of (\ref{FDsadk}) and use the orthonormal property of $\{v_{i,n-1}\}_{i=1}^k$ to obtain for $1\leq i\leq k$
\begin{equation}\label{mh2k}
\begin{array}{l}
\ds v^\top_{i,n-1}\tilde v_{i,n}=v^\top_{i,n-1}v_{i,n-1}+\tau\gamma\bigg( v^\top_{i,n-1}-v^\top_{i,n-1}v_{i,n-1}v_{i,n-1}^\top\\
\ds\qquad\qquad\qquad-2\sum_{j=1}^{i-1}v^\top_{i,n-1}v_{j,n-1}v_{j,n-1}^\top\bigg)\hat H(x_{n-1},v_{i,n-1},l_{n-1})=1.
\end{array}
\end{equation}
We then multiply $\tilde v_{i,n}^\top$ on both sides of the second equation of (\ref{FDsadk}) and apply (\ref{mh2k}) and the orthogonality of $\{v_{i,n-1}\}_{i=1}^k$ to obtain
\begin{equation}\label{mh3k}
\begin{array}{l}
\ds \tilde v_{i,n}^\top\tilde v_{i,n}=\tilde v_{i,n}^\top v_{i,n-1}+\tau\gamma\bigg( \tilde v_{i,n}^\top-\tilde v_{i,n}^\top v_{i,n-1}v_{i,n-1}^\top-2\sum_{j=1}^{i-1}\tilde v_{i,n}^\top v_{j,n-1}v_{j,n-1}^\top\bigg)\hat H(x_{n-1},v_{i,n-1},l_{n-1})\\
\ds\hspace{0.48in} =1+\tau\gamma\bigg( \tilde v_{i,n}^\top-v_{i,n-1}^\top-2\sum_{j=1}^{i-1}(\tilde v_{i,n}-v_{i,n-1})^\top v_{j,n-1} v_{j,n-1}^\top\bigg)\hat H(x_{n-1},v_{i,n-1},l_{n-1}).
\end{array}
\end{equation}
Invoking (\ref{aux1}) in (\ref{mh3k}) leads to
$$\begin{array}{l}
\ds \big|\|\tilde v_{i,n}\|^2-1\big|\leq \tau\gamma\bigg( \|\tilde v_{i,n}-v_{i,n-1}\|+2\sum_{j=1}^{i-1}\|\tilde v_{i,n}-v_{i,n-1}\|\bigg)\|\hat H(x_{n-1},v_{i,n-1},l_{n-1})\|\\[0.2in]
\ds\qquad~\quad\qquad\leq Q\tau^2,~~1\leq i\leq k,~~1\leq n\leq K,
\end{array} $$
which completes the proof.
\end{proof}

\begin{lemma}\label{lem2k}
Suppose (\ref{beta}), (\ref{beta2}) and the Assumption $\mathcal A$ hold and $l_0^2=O(\tau)$, the following estimate holds for $1\leq n\leq K$ and $\tau$ sufficiently small 
$$\|v_{i,n}-\tilde v_{i,n}\|\leq Q\tau^2,~~1\leq i\leq k.$$
Here the positive constant $Q$ is independent from $n$, $K$ and $\tau$. 
\end{lemma}
\begin{proof}
The proof could be performed following that of \cite[Lemma 4.2]{Z3} and is thus omitted. \end{proof}

\section{Error estimate and accuracy improvment}\label{secconv}
In this section we prove error estimates for the numerical discretization (\ref{FDsadk}) to the SSD (\ref{sadk}). Based on the analyzed first-order scheme (\ref{FDsadk}), we then employ the Richardson extrapolation to obtain a second-order approximation.

\subsection{Error estimate of (\ref{FDsadk})}\label{sec41}
we analyze the scheme (\ref{FDsadk}) in the following theorem.
\begin{theorem}\label{thmevk}
Suppose (\ref{beta}), (\ref{beta2})  and the Assumption $\mathcal A$ hold and $l_0^2=O(\tau)$. Then the following estimate holds for $\tau$ sufficiently small 
$$\max_{1\leq n\leq K}\bigg(\|x(t_n)-x_n\|+\sum_{i=1}^k\|v_i(t_n)-v_{i,n}\|\bigg)\leq Q\tau. $$
Here $Q$ is independent from $\tau$, $n$ and $K$.
\end{theorem}
\begin{proof} Let 
\begin{equation}\label{ZZ1}
e^x_n:=x(t_n)-x_n,~~e^{v_i}_{n}:=v_i(t_n)-v_{i,n}
\end{equation}
and we subtract the second equation of (\ref{Refk}) from that of (\ref{FDsadk}) and apply the splitting
$$v_i(t_n)-\tilde v_{i,n}=e^{v_i}_n+(v_{i,n}-\tilde v_{i,n})$$
 to obtain
\begin{equation}\label{th1}
\begin{array}{rl}
\ds e^{v_i}_{n}&\hspace{-0.1in}\ds=e^{v_i}_{n-1}+\tau\gamma\big(\hat H(x(t_{n-1}),v_i(t_{n-1}),l(t_{n-1})-\hat H(x_{n-1},v_{i,n-1},l_{n-1})\big)\\[0.05in]
&\hspace{-0.1in}\ds\quad~~-\tau\gamma \big[ v_i(t_{n-1})v_i(t_{n-1})^\top \hat H(x(t_{n-1}),v_i(t_{n-1}),l(t_{n-1}))\\[0.05in]
&\hspace{-0.1in}\ds\qquad\quad-v_{i,n-1}v_{i,n-1}^\top \hat H(x_{n-1},v_{i,n-1},l_{n-1})\big]\\[0.05in]
&\hspace{-0.1in}\ds\quad~~-2\tau\gamma \sum_{j=1}^{i-1}\big[ v_j(t_{n-1})v_j(t_{n-1})^\top \hat H(x(t_{n-1}),v_i(t_{n-1}),l(t_{n-1}))\\[0.05in]
&\hspace{-0.1in}\ds\qquad\quad-v_{j,n-1}v_{j,n-1}^\top \hat H(x_{n-1},v_{i,n-1},l_{n-1})\big]-(v_{i,n}-\tilde v_{i,n})+\tau R_{n}^{v_i} .
\end{array}
\end{equation}
By (\ref{H0}) we bound the first difference on the right-hand side of (\ref{th1})
\begin{equation}\label{HH}
\begin{array}{l}
\ds \big\|\hat H(x(t_{n-1}),v_i(t_{n-1}),l(t_{n-1}))-\hat H(x_{n-1},v_{i,n-1},l_{n-1})\big\|\\[0.05in]
\ds\qquad=\big\|H(x(t_{n-1}))v_i(t_{n-1})-H(x_{n-1})v_{i,n-1}+O(\tau)\big\|\\[0.05in]
\ds\qquad=\big\|H(x(t_{n-1}))(v_i(t_{n-1})-v_{i,n-1})+(H(x(t_{n-1}))-H(x_{n-1}))v_{i,n-1}+O(\tau)\big\|\\[0.05in]
\qquad\ds\leq Q\|e^{v_i}_{n-1}\|+Q\|e^x_{n-1}\|+Q\tau.
\end{array} 
\end{equation}
To generate errors from other differences on the right-hand side of (\ref{th1}), we should introduce several intermediate terms to split them. For instance, the second difference on the right-hand side of (\ref{th1}) could be split as
$$\begin{array}{l}
\ds \|v_i(t_{n-1})v_i(t_{n-1})^\top \hat H(x(t_{n-1}),v_i(t_{n-1}),l(t_{n-1}))-v_{i,n-1}v_{i,n-1}^\top \hat H(x_{n-1},v_{i,n-1},l_{n-1})\|\\[0.05in]
\ds \quad=\|e^{v_i}_{n-1}v_i(t_{n-1})^\top \hat H(x(t_{n-1}),v_i(t_{n-1}),l(t_{n-1}))+v_{i,n-1}(e^{v_i}_{n-1})^\top \hat H(x_{n-1},v_{i,n-1},l_{n-1})\\[0.05in]
\ds\quad\quad+v_{i,n-1}v_{i,n-1}^\top (\hat H(x(t_{n-1}),v_i(t_{n-1}),l(t_{n-1}))-\hat H(x_{n-1},v_{i,n-1},l_{n-1}))\|\\[0.05in]
\ds \quad\leq Q(\|e^{v_i}_{n-1}\|+\|e^x_{n-1}\|+\tau)
\end{array} $$
where we used (\ref{HH}) and the boundedness of $\hat H$ in the last estimate. The other differences on the right-hand side of (\ref{th1}) could be estimated similarly. We incorporate these estimates in (\ref{th1}) and apply Lemma \ref{lem2k} to obtain
\begin{equation}\label{zz1}
\ds\| e^{v_i}_{n}\|\leq \|e^{v_i}_{n-1}\|+Q\tau\big(\|e^x_{n-1}\|+\|e^v_{n-1}\|\big)+Q\tau^2
\end{equation}
where
\begin{equation} \label{ZZ2}
\|e^v_n\|:=\sum_{j=1}^k\|e^{v_j}_{n}\|.
\end{equation}
We then subtract the first equation of (\ref{Refk}) from that of (\ref{FDsadk}) to obtain
\begin{equation*}
\begin{array}{rl}
\ds e^x_n&\ds\hspace{-0.1in}=e^x_{n-1}+\tau\beta(F(x(t_{n-1}))-F(x_{n-1}))-2\tau\beta\sum_{j=1}^k\big[v_j(t_{n-1})v_j(t_{n-1})^\top F(x(t_{n-1}))\\[0.1in]
&\hspace{-0.1in}\ds\quad\quad-v_{j,n-1}v_{j,n-1}^\top F(x_{n-1})\big]+\tau R^x_n \\[0.05in]
&\hspace{-0.1in}\ds=e^x_{n-1}+\tau\beta(F(x(t_{n-1}))-F(x_{n-1}))-2\tau\beta\sum_{j=1}^k\big[e^{v_j}_{n-1}v_j(t_{n-1})^\top F(x(t_{n-1}))\\[0.1in]
&\hspace{-0.1in}\ds\quad\quad+v_{j,n-1}(e^{v_j}_{n-1})^\top F(x(t_{n-1}))+v_{j,n-1}v_{j,n-1}^\top \big(F(x(t_{n-1}))-F(x_{n-1})\big)\big]+\tau R^x_n.
\end{array}
\end{equation*}
Similar to the above derivations, we apply the Assumption $\mathcal A$, the boundedness of $\|F\|$ and $\|R^x_n\|=O(\tau)$ to find
\begin{equation*}
\ds \|e^x_n\|\leq\|e^x_{n-1}\|+Q\tau\|e^x_{n-1}\|+Q\tau\|e^{v}_{n-1}\| +Q\tau^2.
\end{equation*}
Adding this equation from $n=1$ to $m$ yields
\begin{equation*}
\ds \|e^x_m\|\leq Q\tau\sum_{n=1}^m\|e^x_{n-1}\|+Q\tau\sum_{n=1}^m\|e^{v}_{n-1}\| +Q\tau.
\end{equation*}
Then an application of the discrete Gronwall inequality leads to
\begin{equation}\label{zz2}
\|e^x_{n}\|\leq Q\tau\sum_{m=1}^{n-1}\|e^v_{m}\| +Q\tau ,~~1\leq n\leq K. 
\end{equation}
We invoke this equation in (\ref{zz1}) to obtain
\begin{equation*}
\ds\| e^{v_i}_{n}\|\leq \|e^{v_i}_{n-1}\|+Q\tau\|e^v_{n-1}\|+Q\tau^2\sum_{m=1}^{n-1}\|e^v_{m}\| +Q\tau^2.
\end{equation*}
We then sum up this equation for $1\leq i\leq k$ to get
\begin{equation*}
\ds \| e^{v}_{n}\|\leq \|e^{v}_{n-1}\|+Q\tau\|e^v_{n-1}\|+Q\tau^2\sum_{m=1}^{n-1}\|e^v_{m}\| +Q\tau^2.
\end{equation*}
Adding this equation from $n=1$ to $n_*$ leads to
\begin{equation*}
\ds \| e^{v}_{n^*}\|\leq Q\tau\sum_{n=1}^{n_*}\|e^v_{n-1}\|+Q\tau^2\sum_{n=1}^{n_*}\sum_{m=1}^{n-1}\|e^v_{m}\| +Q\tau\leq Q\tau\sum_{n=1}^{n_*}\|e^v_{n-1}\| +Q\tau.
\end{equation*}
 Then an application of the discrete Gronwall inequality again yields
$$\|e^v_n\|\leq Q\tau,~~1\leq n\leq K, $$ 
  and we combine this with (\ref{zz2}) to obtain the estimate of $\|e^x_n\|$, which completes the proof of this theorem.
\end{proof}
\subsection{A second-order accuracy technique}\label{sec42}
In Section \ref{sec41}, we show that the scheme (\ref{FDsadk}) has the first-order accuracy. A useful approach to get the high-order approximations from the low-order ones is the Richardson extrapolation (see e.g., \cite{Com,Qua}), which is a smart combination of numerical solutions of low-order schemes under different partitions to reach high-order accuracy. A typical and simple example is the second-order Richardson extrapolation. Let 
$$\{x_n,v_{1,n},\cdots,v_{k,n}\}_{n=0}^N$$
 and 
 $$\{\bar x_m,\bar v_{1,m},\cdots,\bar v_{k,m}\}_{m=0}^{2N}$$
  be numerical solutions of the first-order scheme (\ref{FDsadk}) with the mesh numbers $K$ and $2K$, respectively. Then the Richardson extrapolation yields the approximation solution 
  $$\{x^R_n,v^R_{1,n},\cdots,v^R_{k,n}\}_{n=0}^K$$
   of second-order accuracy on the coarse mesh defined as
$$x^R_n=2\bar x_{2n}-x_n,~~v^R_{i,n}=2\bar v_{i,2n}-v_{i,n},~~1\leq i\leq k,~~0\leq n\leq K. $$
The analysis of the second-order accuracy is standard and we refer \cite[Section 9.6]{Qua} for details.
\section{Generalized SSD of non-gradient systems}
In many autonomous dynamical systems there exists no energy $E(x)$ such that $F(x)=-\nabla E(x)$, that is, these systems are non-gradient dynamics. In this case, the following SD is developed in \cite{YinSCM} via the Jacobian $J(x)=\nabla F(x)$ to search for the saddle points of non-gradient systems
\begin{equation}\label{ngsadk}
\left\{
\begin{array}{l}
\ds \frac{dx}{dt} =\bigg(I -2\sum_{j=1}^k v_jv_j^\top \bigg)F(x),\\[0.075in]
\ds \frac{dv_i}{dt}=(I-v_iv_i^\top)\nabla F\cdot v_i -\sum_{j=1}^{i-1}v_jv_j^\top(J(x)+J(x)^\top )v_i,~~1\leq i\leq k.
\end{array}
\right.
\end{equation}
If we again employ the dimer method as (\ref{sadk}) to approximate the multiplication of the Jacobian and the vector for efficient implementation, then the following generalized SSD could be derived from (\ref{ngsadk})
\begin{equation}\label{ngsadksd}
\left\{
\begin{array}{l}
\ds \frac{dx}{dt} =\bigg(I -2\sum_{j=1}^k v_jv_j^\top \bigg)F(x),\\[0.075in]
\ds \frac{dv_i}{dt}=(I-v_iv_i^\top)\hat H(x,v_i,l) \\
\ds\qquad-\sum_{j=1}^{i-1}v_j(v_j^\top \hat H(x,v_i,l)+v_i^\top \hat H(x,v_j,l)),~~1\leq i\leq k,\\
\ds \frac{dl}{dt}=-l.
\end{array}
\right.
\end{equation}

 Compared with the SSD (\ref{sadk}), a symmetrization $v_j^\top \hat H(x,v_i,l)+v_i^\top\hat H(x,v_j,l)$ is used to replace $2v_j^\top \hat H(x,v_i,l)$ in the dynamics of $v_i$ in response to the asymmetry of $\nabla F$. Similar to Section \ref{secnumer}, the corresponding numerical scheme to (\ref{ngsadksd}) reads
 \begin{equation}\label{ngFDsadk}
\left\{
\begin{array}{l}
\ds x_{n} =x_{n-1}+\tau\bigg(I -2\sum_{j=1}^k v_{j,n-1}v_{j,n-1}^\top \bigg)F(x_{n-1}),\\[0.075in]
\ds \tilde v_{i,n}=v_{i,n-1}+\tau\big( I-v_{i,n-1}v_{i,n-1}^\top\big)\hat H(x_{n-1},v_{i,n-1},l_{n-1})\\[0.1in]
\ds\hspace{0.8in}-\tau\sum_{j=1}^{i-1}v_{j,n-1}\big(v_{j,n-1}^\top \hat H(x_{n-1},v_{i,n-1},l_{n-1})\\[0.2in]
\ds\hspace{0.9in}+v_{i,n-1}^\top \hat H(x_{n-1},v_{j,n-1},l_{n-1})\big),~~1\leq i\leq k,\\[0.075in]
\ds \{v_{i,n}\}_{i=1}^k=\text{GS}(\{\tilde v_{i,n}\}_{i=1}^k).
\end{array}
\right.
\end{equation}
We may follow the preceding proofs to analyze the scheme (\ref{ngFDsadk}). However, a key difference that may lead to the failure of recycling the developed ideas and techniques lies in the estimate (\ref{mhback}) of $\tilde v_{m,n}^\top \tilde v_{i,n}$ for $1\leq m< i\leq k$, which is delicate as we require  $O(\tau^2)$ accuracy. Therefore, we reestimate this term for the scheme (\ref{ngFDsadk}) as follows
\begin{equation*}
\begin{array}{l}
\ds \tilde v_{m,n}^\top \tilde v_{i,n}=\tau\big[v_{m,n-1}^\top \hat H(x_{n-1},v_{i,n-1},l_{n-1})-\big(v_{m,n-1}^\top \hat H(x_{n-1},v_{i,n-1},l_{n-1})\\[0.075in]
\ds\hspace{0.3in}+v^\top_{i,n-1}\hat H(x_{n-1},v_{m,n-1},l_{n-1})\big)+\big(\hat H(x_{n-1},v_{m,n-1},l_{n-1})\big)^\top v_{i,n-1}\big]+O(\tau^2)=O(\tau^2)
\end{array}
\end{equation*}
 where we used the observation that the content in $[\cdots]$ in the last-but-one equality is exactly $0$ by virtue of the symmetrization. The other proofs could be performed in parallel to prove first-order accuracy for all variables in the numerical scheme (\ref{ngFDsadk}) of the generalized SSD (\ref{ngsadksd}) for non-gradient systems, and the Richardson extrapolation proposed in Section \ref{sec42} could also be employed to obtain the approximate solutions of second-order accuracy. 
 
\section{Numerical experiments}
In this section, we carry out numerical experiments to substantiate the accuracy of the numerical schemes (\ref{FDsadk}) and (\ref{ngFDsadk}). For applications of these schemes in practical problems, we refer \cite{YinSISC,YinSCM} for various physical examples and detailed discussions. As the exact solutions to the dynamics are not available, numerical solutions computed under $\tau=2^{-13}$ serve as the reference solutions. In the following examples, we set $\beta=\gamma=T=1$ for simplicity and denote the convergence rate by CR. The errors $\|e^x_n\|$ and $\|e^v_n\|$ measured in the experiments are defined in (\ref{ZZ1}) and (\ref{ZZ2}), and we further define the norms
$\|e_n^{R,x}\|:=\|x(t_n)-x^R_n\|$ and $\|e^{R,v}_n\|:=\sum_{j=1}^k\|v_j(t_n)-v^R_{j,n}\|$ 
for the errors of the Richardson extrapolation. In all experiments, $l_0$ is chosen as $\sqrt{\tau}$.

\subsection{First-order scheme for gradient system}\label{sec61}
We consider the SSD (\ref{sadk}) for the stingray function $
 E(x_1,x_2)=x_1^2+(x_1-1)x_2^2
$ \cite{Gra}
and compute its index-1 and index-2 saddle points via scheme (\ref{FDsadk}) with the initial conditions 
$x_0=(1,1)^\top,~~v_0=(0,1)^\top$
 and
  $x_0=(1,1)^\top,~~v_{1,0}=(0,1)^\top,~~v_{2,0}=(1,0)^\top,$
  respectively. Numerical results are presented in Tables \ref{table1:1}-\ref{table1:2}, which demonstrate the first-order accuracy of the numerical scheme (\ref{FDsadk}) as proved in Section \ref{secconv}.
\begin{table}[htb]
\setlength{\abovecaptionskip}{0pt}
\centering
\caption{Convergence of (\ref{FDsadk}) for finding an index-1 saddle point in Example 1.}
\vspace{0.25em}	
\begin{tabular}{ccccc} \cline{1-5}
$1/\tau$& $\max_n\|e^x_n\|$ & CR &  $\max_n\|e^v_n\|$ &CR\\ \cline{1-5}		
$2^5$&	2.60E-02&		&1.91E-02	&\\
$2^6$&	1.23E-02&	1.08&	9.22E-03&	1.05\\
$2^7$&	5.98E-03&	1.05&	4.51E-03&	1.03\\
$2^8$&	2.91E-03&	1.04&	2.20E-03&	1.03\\

				\hline
			\end{tabular}
			\label{table1:1}
		\end{table}

\begin{table}[htb]
\setlength{\abovecaptionskip}{0pt}
\centering
\caption{Convergence of (\ref{FDsadk}) for finding an index-2 saddle point in Example 1.}
\vspace{0.25em}	
\begin{tabular}{ccccc} \cline{1-5}
$1/\tau$& $\max_n\|e^x_n\|$ & CR &  $\max_n\|e^v_n\|$ &CR\\ \cline{1-5}		
$2^5$&	1.50E-02&	&	3.90E-02	&\\
$2^6$&	7.41E-03&	1.02&	1.90E-02&	1.04\\
$2^7$&	3.66E-03&	1.02&	9.30E-03&	1.03\\
$2^8$&	1.79E-03&	1.03&	4.55E-03&	1.03\\
				\hline
			\end{tabular}
			\label{table1:2}
		\end{table}

\subsection{First-order scheme for non-gradient system}\label{sec62}
We consider the following (non-gradient) dynamical system
$$\frac{dx}{dt}=
\left[
\begin{array}{ccc}
 1& 0.5 &0\\
 -0.5& 1 &-0.3\\
 0 &-0.2& 1
\end{array}
\right]x+
\left[
\begin{array}{c}
 (1+(x_1-1)^2)^{-1}\\
(1+(x_2-2)^2)^{-1} \\
 (1+(x_3+1)^2)^{-1}
\end{array}
\right] $$
and use the generalized SSD (\ref{ngFDsadk}) to compute the index-1 and index-2 saddle points of this dynamical system with the initial conditions 
$x_0=(-1,1,0)^\top,~~v_0=(-1,0,0)^\top$
 and
  $x_0=(-1,1,0)^\top,~~v_{1,0}=\frac{1}{\sqrt{2}}(-1,1,0)^\top,~~v_{2,0}=\frac{1}{\sqrt{2}}(1,1,0)^\top,$
respectively. Numerical results are presented in Tables \ref{table2:1}-\ref{table2:2}, which again show the first-order accuracy of the scheme (\ref{ngFDsadk}).

\begin{table}[htb]
\setlength{\abovecaptionskip}{0pt}
\centering
\caption{Convergence of (\ref{ngFDsadk}) for finding an index-1 saddle point.}
\vspace{0.25em}	
\begin{tabular}{ccccc} \cline{1-5}
$1/\tau$& $\max_n\|e^x_n\|$ & CR &  $\max_n\|e^v_n\|$ &CR\\ \cline{1-5}	
$2^5$&	4.95E-02&		&9.32E-03	&\\
$2^6$&	2.50E-02&	0.98&	4.64E-03&	1.00\\
$2^7$&	1.25E-02&	1.00	&2.30E-03	&1.01\\
$2^8$&	6.19E-03&	1.02&	1.13E-03&	1.02\\
				\hline
			\end{tabular}
			\label{table2:1}
		\end{table}

\begin{table}[htb]
\setlength{\abovecaptionskip}{0pt}
\centering
\caption{Convergence of (\ref{ngFDsadk}) for finding an index-2 saddle point.}
\vspace{0.25em}	
\begin{tabular}{ccccc} \cline{1-5}
$1/\tau$& $\max_n\|e^x_n\|$ & CR &  $\max_n\|e^v_n\|$ &CR\\ \cline{1-5}		
$2^5$&	3.00E-02&		&1.02E-02	&\\
$2^6$&	1.50E-02&	1.01&	5.08E-03&	1.00\\
$2^7$&	7.42E-03&	1.01&	2.52E-03&	1.01\\
$2^8$&	3.65E-03&	1.02&	1.24E-03&	1.02\\
				\hline
			\end{tabular}
			\label{table2:2}
		\end{table}

\subsection{Second-order scheme for gradient and non-gradient systems}

We test the convergence rates of the Richardson extrapolation technique proposed in Section \ref{sec42} by the same examples in Sections \ref{sec61}--\ref{sec62} and numerical results are presented in Tables \ref{table3:1}-\ref{table3:4}, which show the second-order accuracy of the Richardson extrapolation.

\begin{table}[htb]
\setlength{\abovecaptionskip}{0pt}
\centering
\caption{Convergence of Richardson extrapolation for finding an index-1 saddle point of the gradient system.}
\vspace{0.25em}	
\begin{tabular}{ccccc} \cline{1-5}
$1/\tau$& $\max_n\|e^{R,x}_n\|$ & CR &  $\max_n\|e^{R,v}_n\|$ &CR  \\ \cline{1-5}	
$2^5$&	1.45E-03&		&5.49E-04	&\\
$2^6$&	3.46E-04&	2.07 &	1.34E-04&	2.03\\ 
$2^7$&	8.43E-05&	2.04 &	3.31E-05&	2.02 \\
$2^8$&	2.08E-05&	2.02 &	8.22E-06&	2.01 \\
				\hline
			\end{tabular}
			\label{table3:1}
		\end{table}

\begin{table}[htb]
\setlength{\abovecaptionskip}{0pt}
\centering
\caption{Convergence of Richardson extrapolation for finding an index-2 saddle point of the gradient system.}
\vspace{0.25em}	
\begin{tabular}{ccccc} \cline{1-5}
$1/\tau$& $\max_n\|e^{R,x}_n\|$ & CR &  $\max_n\|e^{R,v}_n\|$ &CR\\ \cline{1-5}		
$2^5$&	3.39E-04&		&9.79E-04	&\\
$2^6$&	8.32E-05&	2.03 &	2.41E-04&	2.02\\ 
$2^7$&	2.06E-05&	2.01 &	5.97E-05&	2.01 \\
$2^8$&	5.13E-06&	2.01 &	1.49E-05&	2.01 \\
				\hline
			\end{tabular}
			\label{table3:2}
		\end{table}
		
\begin{table}[htb]
\setlength{\abovecaptionskip}{0pt}
\centering
\caption{Convergence of Richardson extrapolation for finding an index-1 saddle point of the non-gradient system.}
\vspace{0.25em}	
\begin{tabular}{ccccc} \cline{1-5}
$1/\tau$& $\max_n\|e^{R,x}_n\|$ & CR &  $\max_n\|e^{R,v}_n\|$ &CR \\ \cline{1-5}	
$2^5$&	9.54E-04&		&1.43E-04	&\\
$2^6$&	2.45E-04&	1.96&	3.52E-05&	2.02\\
$2^7$&	6.20E-05&	1.98&	8.71E-06&	2.01\\
$2^8$&	1.56E-05&	1.99&	2.17E-06&	2.01\\
				\hline
			\end{tabular}
			\label{table3:3}
		\end{table}

\begin{table}[htb]
\setlength{\abovecaptionskip}{0pt}
\centering
\caption{Convergence of Richardson extrapolation for finding an index-2 saddle point of the non-gradient system.}
\vspace{0.25em}	
\begin{tabular}{ccccc} \cline{1-5}
$1/\tau$& $\max_n\|e^{R,x}_n\|$ & CR &  $\max_n\|e^{R,v}_n\|$ &CR\\ \cline{1-5}		
$2^5$&	1.43E-04&		&1.53E-04	&\\
$2^6$&	3.55E-05&	2.01&	3.86E-05&	1.99\\
$2^7$&	8.87E-06&	2.00	&9.69E-06&	1.99\\
$2^8$&	2.21E-06&	2.00	&2.42E-06&	2.00\\
				\hline
			\end{tabular}
			\label{table3:4}
		\end{table}

\section{Conclusions}
Finding the saddle points of complicated systems has attracted an increasing interest in recent decades. In particular, the SSD serves as an efficient numerical algorithm to compute any-index saddle points and has been widely used to construct the solution landscapes of varied energy and dynamical systems. In this paper we prove the boundedness of the exact solutions and optimal-order error estimates of the numerical discretization to the SSD with respect to the time step size. We overcome the main difficulties of dealing with the local Lipschitz assumptions, the dimer approximation, and the strong nonlinearity of the SSD. We further employ the Richardson extrapolation to obtain the approximate solution with second-order accuracy. The derived analysis and numerical results provide mathematical and numerical supports for the computations of saddle points. In future works, we will investigate how to relax or eliminate the restrictions like (\ref{beta}) on the parameters to improve the analysis.

\section*{Acknowledgments} 
This work was partially supported by the National Natural Science Foundation of China No.~12050002 and 21790340; the National Key R\(\&\)D Program of China No.~2021YFF1200500; the International Postdoctoral Exchange Fellowship Program (Talent-Introduction Program) No.~YJ20210019; the China Postdoctoral Science Foundation No.~2021TQ0017 and 2021M700244.


\end{document}